\documentclass[11pt]{article}

\usepackage[a4paper, left=3cm, top=3cm, right=3cm, bottom=3cm ]{geometry}
\usepackage{authblk}
\usepackage[onehalfspacing]{setspace}

\usepackage[T1]{fontenc}

\usepackage[bookmarks=closed]{hyperref}
\usepackage{breakurl}

\usepackage{amsmath,amsfonts}
\usepackage{algorithmic}
\usepackage{algorithm}
\usepackage{array}
\usepackage[caption=false]{subfig}
\usepackage{textcomp}
\usepackage{stfloats}
\usepackage{hyperref}
\usepackage{url}
\usepackage{verbatim}
\usepackage{siunitx}
\usepackage{graphicx}
\usepackage{amssymb, amsthm}
\usepackage{cite}

\usepackage{enumitem}
\usepackage{multirow}

\usepackage{todonotes}

\newtheorem{theorem}{Theorem}[section]
\newtheorem{lemma}[theorem]{Lemma}

\newtheorem{definition}[theorem]{Definition}

\usepackage{siunitx}

\usepackage{mathtools}
\usepackage{booktabs}
\usepackage{dsfont}
\usepackage{pgfplots}
\usetikzlibrary{intersections, pgfplots.fillbetween}
\definecolor{mycolor}{RGB}{85,125,250}
\definecolor{gcolor}{RGB}{0.01,0.199,0.1}

\tikzset{declare function = {
                aux(\x)= (\x <= 0) * (0) +
                          and(\x > 0, \x < 1) * ((3*(\x-1)^2-2*(\x-1)^3)/( 3*(\x-1)^2-2*(\x-1)^3 + (3*\x^2-2*\x^3))) +
                          (\x >= 1) * (1)
               ;
               v(\x)= (abs(\x) <= 1/3) * (1) +
                          and(abs(\x) >= 1/3, abs(\x) <= 2/3) * (cos(pi/2)*aux(3*abs(\x)-1))    +
                          (abs(\x) >= 2/3) * (0)
               ;
               left(\x)= (\x >= -1/3) * (0) +
                          and(\x <= -1/3, \x >= -2/3) * (cos(pi/2)*v(3*abs(\x)-1))    +
                          (\x <= -2/3) * (0)
               ;
               right(\x)= (\x <= 1/3) * (0) +
                          and(\x >= 1/3, \x <= 2/3) * (cos(pi/2)*v(3*abs(\x)-1))    +
                          (\x >= 2/3) * (0)
               ;
               mid(\x)= (abs(\x) <= 1/3) * (1) + (abs(\x) >= 1/3) * (0)
               ;
               left_adap(\x)= (\x >= -1/3) * (0) +
                          and(\x <= -1/3, \x >= -2/3) * (cos(pi/2)*v(3*abs(\x)-1))    +
                          (\x <= -2/3) * (0)
               ;
               right_adap(\x)= (\x <= 1/3) * (0) +
                          and(\x >= 1/3, \x <= 2/3) * (cos(pi/2)*v(3*abs(\x)-1))    +
                          (\x >= 2/3) * (0)
               ;
               mid_adap(\x)= (abs(\x) <= pi/18) * (1) + (abs(\x) >= pi/18) * (0)
               ;
               }}

\newcommand{\abs}[1]{\lvert#1\rvert}
\newcommand{\norm}[1]{\lVert#1\rVert}

\newcommand{\Fo}{\mathbf{F}}
\newcommand{\Ko}{\mathbf{K}}

\DeclareMathOperator{\ran}{ran}       

\newcommand{\Base}{\boldsymbol \Psi}
\newcommand{\synthesis}{\Base^*}
\newcommand{\analysis}{\Base}
\newcommand{\boldframe}{\boldsymbol \psi}

\newcommand{\winkel}{\phi}

\newcommand{\limset}{\Omega}

\newcommand{\inner}[1]{\left\langle#1\right\rangle}
\newcommand{\set}[1]{\{#1\}}

\newcommand{\tik}{\mathcal{T}}
\newcommand{\dist}{\mathcal{F}_0}
\newcommand{\fun}{\mathcal{F}}
\newcommand{\gun}{\mathcal{G}}
\newcommand{\reg}{\mathcal{R}}

\newcommand{\coeff}{\theta}
\newcommand{\signal}{u}
\newcommand{\aux}{w}
\newcommand{\data}{v}

\newcommand{\C}{\mathbb{C}}
\newcommand{\R}{\mathbb{R}}
\newcommand{\N}{\mathbb{N}}
\newcommand{\Z}{\mathbb{Z}}
\newcommand{\sph}{\mathbb{S}}
\newcommand{\dom}{\mathcal{D}}

\newcommand{\al}{\alpha}

\newcommand{\la}{\lambda}
\newcommand{\La}{\Lambda}

\newcommand{\vis}{\text{vis}}
\newcommand{\inv}{\text{inv}}

\AtBeginDocument{
  \let\div\relax
  \DeclareMathOperator{\div}{div}
}
\newcommand*\diff{\mathop{}\!\mathrm{d}}

\DeclareMathOperator*{\argmin}{arg\,min}
\DeclareMathOperator{\supp}{supp}

\newcommand{\No}{\mathbf{N}}

\usepackage{soul}
\colorlet{lred}{red!40}
\colorlet{lgreen}{green!40}
\colorlet{lblue}{blue!40}
\definecolor{bananamania}{rgb}{0.98, 0.91, 0.71}

\numberwithin{equation}{section}
\numberwithin{theorem}{section}
\numberwithin{figure}{section}

\allowdisplaybreaks

\title{Data-proximal complementary $\ell^1$-TV reconstruction for limited data CT}

\date{\today}

\author{Simon Göppel}

\affil{Department of Mathematics, University of Innsbruck\authorcr
Technikerstrasse 13, 6020 Innsbruck, Austria\authorcr
E-mail:  \texttt{simon.goeppel@uibk.ac.at}
 }
 
\author{Jürgen Frikel}

\affil{Department of Computer Science and Mathematics, OTH Regensburg\authorcr Galgenbergstra{\ss}e 32, 93053 Regensburg, Germany\authorcr
E-mail:  \texttt{juergen.frikel@oth-regensburg.de}
 }

\author{Markus Haltmeier}

\affil{Department of Mathematics, University of Innsbruck\authorcr
Technikerstrasse 13, 6020 Innsbruck, Austria\authorcr
E-mail:  \texttt{markus.haltmeier@uibk.ac.at}
 }

\begin{document}

\maketitle

\begin{abstract} 
In a number of tomographic applications, data cannot be fully acquired, resulting in a severely underdetermined image reconstruction. In such cases, conventional methods lead to reconstructions with significant artifacts. To overcome these artifacts, regularization methods are applied that incorporate additional information.  An important example is TV reconstruction, which is known to be efficient at compensating for missing data and reducing reconstruction artifacts. At the same time, however, tomographic data is also contaminated by noise, which poses an additional challenge. The use of a single regularizer must therefore account for both the missing data and the noise.  However, a particular regularizer may not be ideal for both tasks. For example, the TV regularizer is a poor choice for noise reduction across multiple scales, in which case $\ell^1$ curvelet regularization methods are well suited. To address this issue, in this paper we introduce a novel variational regularization framework that combines the advantages of different regularizers. The basic idea of our framework is to perform reconstruction in two stages, where the first stage mainly aims at accurate reconstruction in the presence of noise, and the second stage aims at artifact reduction. Both reconstruction stages are connected by a data proximity condition. The proposed method is implemented and tested for limited-view CT using a combined curvelet-TV approach.   We define and implement a curvelet transform adapted to the limited-view problem and illustrate the advantages of our approach in numerical experiments.

\end{abstract}

\section{Introduction}

Limited data computed tomography (CT) is a prerequisite for a wide range of applications such as digital breast tomosynthesis, dental tomography and non-destructive testing. In this case, the available data is only a subset of the full data that would be required to uniquely identify the scanned object. Due to the lack of available scans, certain image features are invisible and important information may be obscured by artifacts generated during reconstruction \cite{Quinto93,quinto2017artifacts}. Although the characterization of limited view artifacts has been well researched \cite{frikel2013characterization,frikelquinto2016,Borg2018}, effective artifact reduction or compensation for missing data is still a challenge. This is even more true when the tomographic data is noisy, which creates additional hurdles.

Mathematically, limited-data CT can be written as an inverse problem of the form 
\begin{equation} \label{eq:ip}
	 \data^\delta = \No_\delta ( \Ko_\limset  \signal )  \,,
\end{equation}
where  $\signal \in L^2(\R^2)$ is the unknown image to be recovered, $\Ko_\limset$ denotes the Radon transform with restricted angular range $\limset \subseteq \sph^1$ and $\No_\delta$ describes the noise in the data parameterized by the noise level $\delta >0$. While the inverse problem of recovering an image from CT measurements with complete noisy data is already ill-posed \cite{natterer2001mathematics}, the reconstruction problem for incomplete data is severely under-determined. Direct methods such as filtered back projection (FBP) are sensitive to noise and do not handle missing data well, leading to typical limited data artifacts.  To account for noise and missing data, further information that is available about the object to be recovered must be incorporated. Specific methods are therefore required that can both reliably remove noise and avoid artifacts caused by limited data.


\subsection{Variational regularization}

One of the most successful approaches to problems of the form \eqref{eq:ip} is variational regularization \cite{benning2018modern,scherzer2009variational}, in which a stable and robust solution $\signal_\al^\delta \in L^2(\R^2)$ is determined as  minimizer of 
\begin{equation} \label{eq:tik}
 	\tik_\alpha(\signal, \data^\delta) 
	= \frac{1}{2} \norm{\Ko_\limset \signal - \data^\delta}^2    +  \al \reg(\signal) \,.
\end{equation} 
Here  $\reg \colon L^2(\R^2)  \to \R \cup \{\infty\}$ is a suitable regularizer incorporating prior information about the image  to be recovered and $\norm{\Ko_\limset \signal - \data^\delta}^2/2$ is the least squares data fitting functional.  The variational approach offers great flexibility. In particular, it can be adapted to the forward problem, the signal class, and the noise. For example, total variation $ \reg(\signal) = \abs{\signal}_{\rm TV} $ has been shown to be a good prior to complete missing data \cite{persson2001total,velikina2007limited,sidky2008image,wang2017reweighted}. On the other hand, the $\ell^1$-norm  $ \reg(\signal) = \norm{\analysis  \signal}_1$ of wavelet or curvelet coefficients $\analysis  \signal$ has been shown to be statistically optimal for Radon inversion from complete data \cite{candes2002recovering}. On the downside, the mono-scale nature of the total variation does not lead to an optimal reconstruction in the presence of noise \cite{haltmeier2022variational} and  $\ell^1$-reconstructions hardly account for missing data \cite{sahiner1993limited,rantala2006wavelet}.  

The individual advantages and disadvantages of specific regularizers have led to so-called hybrid methods that combine two different regularizers within the variational regularization framework \eqref{eq:tik}. For example, hybrid $\ell^1$-TV methods \cite{vandeghinste2012combined,kai2020moreau} use the regularizer $\reg(\signal) = \alpha \abs{\signal}_{\rm TV} + \beta \norm{\analysis  \signal}_1$. Given the above strengths and limitations of each individual regularizer, this is particularly attractive for CT with noisy limited data. However, the single hybrid regularizer must again account for both, the limited data and the noise, which is a challenging task.  Unfortunately, a fixed hybrid regularizer cannot fully avoid the drawbacks of the individual terms.  For example, the TV term still leads to over or under smoothing of certain scales in the visible range, while the curvelet part  still tries to suppress intensity values of invisible coefficients. To avoid these negative impacts, it is necessary to adapt each regularizer to its actual purpose.
 
 \subsection{Main contribution}
 
In this paper, we present a novel complementary $\ell^1$-TV algorithm that addresses both the limited data problem and the noise reduction problem. It is based on a modified variational regularization approach that selects a regularizer for each of the two tasks and combines them in a  synergetic way through data-proximity. More precisely, let   $\synthesis \colon  \ell_2 (\Lambda) \to L^2(\R^2)$  denote the synthesis operator of some framewith index set $\Lambda$. The proposed iterative reconstruction method generates two reconstructions $\coeff \in  \ell_2 (\Lambda)$ and $\signal \in L^2(\R^2)$ by alternately solving 
\begin{align*} 
     &\min_\coeff 
    \norm{\Ko_\limset (\synthesis \coeff ) - \data^\delta}^2/2  
    + \alpha \norm{\coeff}_1  
    +  \mu \norm{ \Ko_\limset( \signal - \synthesis \coeff)}^2/2 \\
     &\min_\signal
    \reg(\signal) +
     \mu \norm{ \Ko_\limset( \signal - \synthesis \coeff)}^2/2  \,. 
\end{align*}
Here, the auxiliary reconstruction $\synthesis \coeff$ targets a noise-suppressing reconstruction addressed by the sparsity term $\norm{\coeff}_1$.  The primary reconstruction $\signal$ implicitly performs data completion by updating $\synthesis(\coeff)$ based on the regulariser $\reg(\signal)$. A key element is the coupling of the two reconstructions, which requires that $\norm{ \Ko_\limset( \signal - \synthesis \coeff)}^2$ is small, which we will refer to as data proximity.   As a result, both $\signal$ and $\synthesis \coeff$ approximately give the data $\data^\delta$. There are many possible solutions due to ill-poseness, and the specific regularisers allow $\signal$ and $\synthesis \coeff$ to be significantly different.

Note that our method is very different from post-processing an original reconstruction. In the latter case, the data proximity term $\norm{\Ko_\limset (\signal - \synthesis \coeff)}^2$ is replaced by a proximity $\norm{\signal - \synthesis \coeff }^2$ in reconstruction space, which forces $\signal$ to be close to $\synthesis(\coeff)$, making artefacts difficult to remove. We also note that our concept is applicable to any image reconstruction problem with limited data, and that we focus on CT with limited data for the sake of clarity. In addition, we propose several variations of the data-proximity coupling, which will be discussed later in the manuscript.

\section{Background}

Throughout this article, we will use the following notation. The Fourier transform of a function $\signal \in L^2 (\R^2)$ is denoted by $ \Fo \signal$, where $\Fo \signal (\xi) \triangleq \int_{\R^2} \signal(x) e^{-i\inner{\xi, x}}\diff x$ for integrable functions and extended  to $L^2 (\R^2)$ by continuity. We write $\signal^\ast (x) \triangleq \overline{\signal(-x)}$, where $\overline{z}$ denotes the complex conjugate of $z\in\C$.  Recall that the Fourier transform converts convolution into multiplication. In particular, for $\signal,w \in L^2(\R^2)$, with $\Fo \signal \in L^\infty (\R^2)$, the convolution $\signal \ast w \in L^2(\R^2)$ is well-defined and given by $\signal \ast w \ = \Fo^{-1} ( (\Fo \signal) \cdot (\Fo w ) )$.  Furthermore, we write $\Fo_2 \signal$ for the Fourier transform of $\signal \in L^2 (\sph^1\times \R)$ with respect to the second argument.

\subsection{The Radon transform}

The Radon transform with full-angular range  maps any function  $\signal \in L^1(\R^2) \cap L^2(\R^2)$ to the line integrals
\begin{equation*}
    \Ko \signal (\omega, s) \triangleq \int_{\omega^\perp} \signal (x + s \omega) \diff x 
    \quad \text{ for } (\omega, s)  \in \sph^1 \times \R \,.
\end{equation*}
Here $\sph^1 = \set{\omega \in \R^2 \mid \norm{\omega}=1}$,  and any line of integration $\{x \in \R^2 \mid \inner{\omega, x} =s \}$  is described by a unit normal vector $\omega \in \sph^1$  and oriented distance $s$ from the origin.   The Radon transform  can extended to an unbounded densely  defined closed operator $\Ko \colon \dom(\Ko) \subseteq L^2(\R^2)  \to L^2(\sph^1 \times \R)$ with domain $\dom(\Ko) \triangleq \{ \signal \in L^2(\R^2) \mid \norm{\cdot}^{-1/2} \Fo \signal \in L^2(\R^2) \}$; see~\cite{smith1977practical}.


 \begin {lemma}[Fourier slice theorem] \label{lem:fourier-slice}
 For all  $\signal  \in \dom(\Ko) $ we have $\Fo_2 \left( \Ko \signal \right) (\omega, \sigma) = \Fo \signal (\sigma \omega) $.
 \end{lemma}

Opposed to the full data case, in limited data  CT, the  Radon transform is only known on a certain subset. Equivalently, we may model limited view data with a  binary  mask as we will do  here. For any  subset  $A \subseteq    \sph^1 \times \R$ we  denote by  $\chi_A$ the  indicator  function  defined  by $\chi_A(\omega,s ) = 1$ if  $(\omega,s)  \in A$ and $\chi_A(\omega,s) = 0$ otherwise.

\begin{definition}
For $\limset \subseteq \sph^1$  we define the  limited-angle Radon transform as 	
\begin{equation*}
	\Ko_\limset  \colon \dom(\Ko_\limset) \subseteq L^2(\R^2) \to L^2(\sph^1 \times \R)  \colon \signal  \mapsto \chi_{\limset \times \R}  \cdot  (\Ko \signal) \,.
\end{equation*}
\end{definition}

The Fourier slice theorem  states that for any $\omega \in \sph^1$, the Fourier transform  of the Radon transform of some function in the second component equals the Fourier transform of that function along the Fourier slice  $\set{ \sigma \omega \mid \sigma \in \R}$.          
In particular, limited angle CT data is in one-to-one correspondence with the Fourier transform $\Fo \signal$ restricted  to the set $W_\limset \triangleq \set{  \sigma  \omega  \mid  \sigma \in \R  \wedge \omega \in \limset } $. We will call $W_\limset$ the visible wavenumber set, as only Fourier coefficients for wave numbers in $W_\limset$ are provided by the data. Accordingly, we call $\R^2 \setminus W_\limset$ the invisible wavenumber set.  We see that  if  $\R^2 \setminus W_\limset$ has non-vanishing  measure, then $\Ko_\limset$ has non-vanishing kernel consisting of all  functions $\signal \in \dom (\Ko_\limset) =  \dom(\Ko) \cap W_\limset$ with $\supp \Fo u \subseteq W_\limset$.

In limited view CT the set $W_\limset$ forms a wedge, whereas in the sparse view case the set $W_\limset$ forms a fan; see  the left two images in Figure~\ref{fig:setup}.

 \begin{figure}[tbh!]
\centering
  \begin{tikzpicture}[scale=1.4]
	\filldraw [fill=gcolor, draw=black, opacity=0.2] (1.5,1.5) rectangle (-1.5,-1.5);	
	
	\draw[xshift=0cm,yshift=0cm,dotted,red,name path=A] (-1.5, -1.5) -- (1.5, 1.5);
	\draw[xshift=0cm,yshift=0cm,dotted,red,name path=B] (-1.5, 1.5) -- (1.5, -1.5);
	\draw[xshift=0cm,yshift=0cm,dotted,red,name path=C] (-0.7, -1.5) -- (0.7, 1.5);
	\draw[xshift=0cm,yshift=0cm,dotted,red,name path=D] (-0.7, 1.5) -- (0.7, -1.5);
	\tikzfillbetween[of=B and A]{color=mycolor, opacity=1.0};
         \tikzfillbetween[of=A and C]{color=mycolor, opacity=1.0};
         \tikzfillbetween[of=B and D]{color=mycolor, opacity=1.0};

	\draw [black, dashed, xshift=0cm, domain=0:360] plot(\x:1);

	\draw[xshift=0cm,yshift=0cm,-latex,name path=C] (-1.5, 0) -- (1.7, 0);
	\draw[xshift=0cm,yshift=0cm,-latex,name path=D] (0, -1.5) -- (0, 1.7);

        \draw [<->, xshift=0cm, domain=-64:64] plot(\x:1.25);
		
  \end{tikzpicture}
  \quad
   \begin{tikzpicture}[scale=1.4]
	\filldraw [fill=gcolor, draw=black, opacity=0.2] (1.5,1.5) rectangle (-1.5,-1.5);	
	0.2375
	\draw[line width=0.5mm, xshift=0cm,yshift=0cm,mycolor,name path=A0] (-1.5, -1.5) -- (1.5, 1.5);
	\draw[line width=0.5mm, xshift=0cm,yshift=0cm,mycolor,name path=A1] (-1.5, -1.0898) -- (1.5, 1.0898);
	\draw[line width=0.5mm, xshift=0cm,yshift=0cm,mycolor,name path=A2] (-1.5, -0.76428) -- (1.5, 0.76428);
	\draw[line width=0.5mm, xshift=0cm,yshift=0cm,mycolor,name path=A3] (-1.5, -0.4874) -- (1.5, 0.4874);
	\draw[line width=0.5mm, xshift=0cm,yshift=0cm,mycolor,name path=A4] (-1.5, -0.2375) -- (1.5, 0.2375);
	\draw[line width=0.5mm, xshift=0cm,yshift=0cm,mycolor,name path=A5] (-1.5, -0.0) -- (1.5, 0.0);
	\draw[line width=0.5mm, xshift=0cm,yshift=0cm,mycolor,name path=A6] (-1.5, 0.2375) -- (1.5, -0.2375);
	\draw[line width=0.5mm, xshift=0cm,yshift=0cm,mycolor,name path=A7] (-1.5, 0.4874) -- (1.5, -0.4874);
	\draw[line width=0.5mm, xshift=0cm,yshift=0cm,mycolor,name path=A8] (-1.5, 0.76428) -- (1.5, -0.76428);
	\draw[line width=0.5mm, xshift=0cm,yshift=0cm,mycolor,name path=A9] (-1.5, 1.0898) -- (1.5, -1.0898);
	\draw[line width=0.5mm, xshift=0cm,yshift=0cm,mycolor,name path=A10] (-1.5, 1.5) -- (1.5, -1.5);
	\draw[line width=0.5mm, xshift=0cm,yshift=0cm,mycolor,name path=A11] (-1.0898, -1.5) -- (1.0898, 1.5);
	\draw[line width=0.5mm, xshift=0cm,yshift=0cm,mycolor,name path=A12] (-0.76428, -1.5) -- (0.76428, 1.5);
	\draw[line width=0.5mm, xshift=0cm,yshift=0cm,mycolor,name path=A13] (-0.4874, -1.5) -- (0.4874, 1.5);
	\draw[line width=0.5mm, xshift=0cm,yshift=0cm,mycolor,name path=A14] (-0.2375, -1.5) -- (0.2375, 1.5);
	\draw[line width=0.5mm, xshift=0cm,yshift=0cm,mycolor,name path=A15] (-0.0, -1.5) -- (0, 1.5);
	\draw[line width=0.5mm, xshift=0cm,yshift=0cm,mycolor,name path=A16] (0.2375, -1.5) -- (-0.2375, 1.5);
	\draw[line width=0.5mm, xshift=0cm,yshift=0cm,mycolor,name path=A17] (0.4874, -1.5) -- (-0.4874, 1.5);
	\draw[line width=0.5mm, xshift=0cm,yshift=0cm,mycolor,name path=A18] (0.76428, -1.5) -- (-0.76428, 1.5);
	\draw[line width=0.5mm, xshift=0cm,yshift=0cm,mycolor,name path=A19] (1.0898, -1.5) -- (-1.0898, 1.5);

	\draw[xshift=0cm,yshift=0cm,-latex,name path=C] (-1.5, 0) -- (1.7, 0);
	\draw[xshift=0cm,yshift=0cm,-latex,name path=D] (0, -1.5) -- (0, 1.7);

	\draw [black, dashed, xshift=0cm, domain=0:360] plot(\x:1);

		
  \end{tikzpicture}
  \quad 
         \includegraphics[width=0.27\columnwidth]{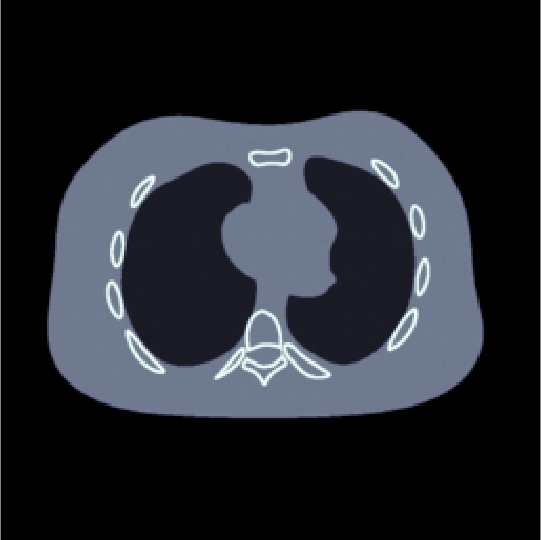}
  \caption{Left: Visible wavenumbers (blue) for limited view data covering $\ang{130}$. Middle: Visible wavenumbers (blue) for sparse angular sampling using 20 angles. Right: Original NCAT phantom used for the numerical simulations shown below.}
  \label{fig:setup}
\end{figure}
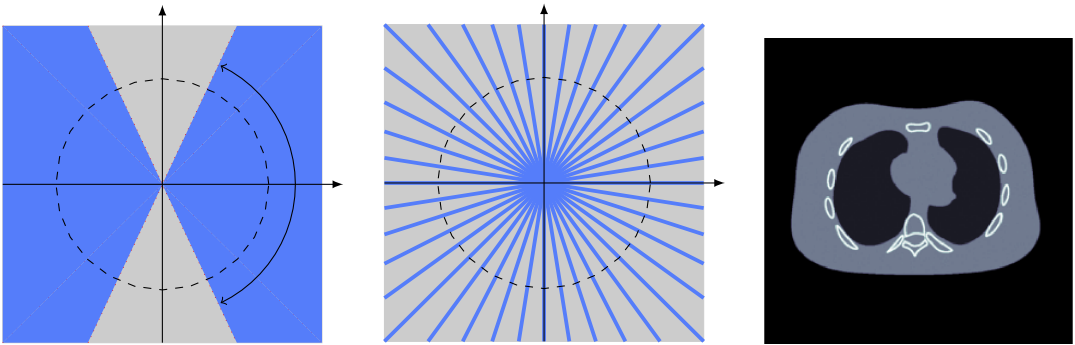

\subsection{Frames and TI-frames}

We frequently use that the desired image  $\signal$ has a sparse representation o approximation in a suitable frame. In particular, we work with curvelet frames, which give an optimal sparse representation of cartoon-like images \cite{candes2002recovering}. The same is true for shearlets \cite{kutyniok2011compactly}. Curvelets and shearlets form frames of $L^2(\R^2)$, and this section provides some necessary background. 

\subsubsection{Translational-invariant (TI) frames}

Let $I$ be an at most countable index set. A family $(\boldframe_i)_{i \in I}$ in $L^2(\R^2)$ is called a translation invariant frame (TI-frame) for $L^2(\R^2)$ if $\Fo \boldframe_i \in L^\infty (\R^2)$ for all $i \in I$ and from some constants $A,B>0$ we have  
\begin{equation}\label{eq:TI-frame}
    \forall \signal \in L^2(\R^2) \colon \quad A\norm{\signal}^2 \leq \sum_{i \in I} \norm{\boldframe_i \ast \signal}^2 \leq B \norm{\signal}^2 \,,
\end{equation}
A TI-frame is called tight if $A=B=1$.  From $ \boldframe_i \ast \signal = \Fo^{-1} ( (\Fo \boldframe_i) \cdot (\Fo \signal)) $  and Plancherel's theorem we  get $\norm{ \boldframe_i \ast \signal}^2  = 2\pi \int_{\R^2} \abs{ \Fo \boldframe_i}^2 \abs{\Fo \signal}^2$. The right inequality  in  \eqref{eq:TI-frame} thus implies  $(\boldframe_i\ast \signal)_{i \in I} \in \ell^2 (I, L^2(\R^2))$. 

Along with TI-frames, we will make use of the  TI-analysis and TI-synthesis operators  respectively,  which are defined by
        \begin{align*}
            &\analysis \colon L^2(\R^2) \to \ell^2(I,L^2(\R^2)) \colon  \signal \mapsto (\boldframe_i \ast \signal )_{i \in I}
 \\
            &\synthesis \colon \ell^2(I,L^2(\R^2)) \to L^2(\R^2) \colon  (\coeff_i)_{i  \in I}\mapsto \sum_{i  \in I}  \boldframe^\ast_i \ast \coeff_i  \,. 
 \end{align*}
Note that the TI-analysis operator and the TI-synthesis operator are the adjoint of each other. The composition $\synthesis\analysis$ is known as the TI-frame operator.  Using the definition of the TI-analysis operator we can rewrite the frame condition \eqref{eq:TI-frame} as $A \norm{\signal}^2 \leq \norm{\analysis \signal}^2\leq B \norm{\signal}^2 $ for $\signal \in L^2(\R^2)$. The right inequality  in \eqref{eq:TI-frame} states that the TI-analysis operator $\Base$ is a well-defined  bounded linear operator. The left inequality states that $\Base$ is bounded from below, that is, the pseudo-inverse $\analysis^\ddagger \triangleq (\synthesis\analysis)^{-1} \Base^*$ is continuous.

See \cite{mallat2008wavelet} for general background on TI-frames and \cite{goppel2023translation,parhi2023sparsity,coifman1995translation} for TI-frames in the context of inverse problems.

\subsubsection{Regular frames}

Regular frames use inner products instead of convolutions as in TI-frames for defining coefficients.   
Let $\La$ be an at most countable index set. A family $(\boldframe_\la)_{\la\in\La}$ in $L^2(\R^2)$ is called a frame for $L^2(\R^2)$ if 
\begin{equation}\label{eq:frame}
    \forall \signal \in L^2(\R^2) \colon \quad A\norm{\signal}^2 \leq \sum_{\la\in\La} \abs{\inner{\boldframe_\la , \signal}}^2 \leq B \norm{\signal}^2,
\end{equation}
for some $A,B>0$. A frame is called tight if $A=B=1$. In some sense the TI-frame can be seen as a frame with index $I \times \R^2$.  Note however that clearly the TI-frame  not a regular frame because  $I \times \R^2$ is uncountable. Similar to the TI case, the analysis and synthesis operators of a regular frame are defined by
        \begin{align*}
            &\analysis \colon L^2(\R^2) \to \ell^2(\La) \colon
            \signal \mapsto (\inner{\boldframe_\la , \signal})_{\la  \in \La}
            \\
            &\synthesis \colon \ell^2(\La) \to L^2(\R^2) \colon  
            (\coeff_\la)_{\la  \in \La}\mapsto \sum_{\la  \in \La} \boldframe^\ast_\la \, \coeff_\la   
\end{align*}
 and the composition $\synthesis\analysis$ is the frame operator.

Under suitable regularity assumptions \cite{daubechies1992ten,mallat2008wavelet}, a regular frame with index set $I \times \Z^2$ can be obtained from a TI-frame with index set $I$ by discretizing the convolution in \eqref{eq:TI-frame}.  For multiscale systems such as wavelets of curvelets, the associated $I$-dependent subsampling destroys translation invariance, which can lead to degraded performance and reconstruction. The advantages of the TI-frames over regular frames have been investigated in \cite{coifman1995translation} for plain denoising and in \cite{goppel2023translation} for general inverse problems.

 \subsection{Variational image reconstruction}

A practically successful and theoretically well analyzed  method  for solving \eqref{eq:ip} is  variational regularization \cite{benning2018modern,scherzer2009variational}. Here, the available prior information is incorporated by a regularization functional $\reg \colon L^2(\R^2) \to \R \cup \{ \infty \}$ and an approximate  image is recovered by  minimizing the Tikhonov functional $\tik_\alpha (\signal, \data^\delta) = \norm{\Ko_\limset \signal - \data^\delta}^2/2 + \al \reg(\signal)$ with respect to $\signal$; see  \eqref{eq:tik}.    

Variational regularization  is well-posed, stable and convergent   in the following sense: (i) $\tik_{\alpha} (\cdot, \data^\delta)$ has  a minimizer $\signal_\alpha^\delta$; (ii) minimizers  depend continuously on data  $ \data^\delta$;  (iii) if $\norm{\data - \data^\delta} \leq \delta $  with $\data \in \ran(\Ko_\limset)$ and $\alpha = \alpha (\delta) $ is selected properly then $\signal_\alpha^\delta$ converges (as $\delta \to 0$) to an $\reg$-minimizing solution of $\Ko_\limset \signal = \data$ defined by 
\begin{equation} \label{eq:Rmin}
	\min_\signal  \reg(\signal) \quad \text{ such that } \Ko_\limset \signal = \data \,.  
\end{equation}
These  properties hold true  under the assumption that  $\reg$ is convex, weakly lower semicontinuous and coercive  \cite{scherzer2009variational}. 
The characterization \eqref{eq:Rmin} of the limiting  solutions reveals two separate tasks to be performed by the regularizer:  Besides noise-robust reconstructions via minimization of the Tikhonov functional, it also serves as criteria for selecting a particular solution in the limit of noise-free data.  Obviously, it is difficult to optimally perform both tasks with a single regularizer.
Note that the  selection of a particular solution via  \eqref{eq:Rmin}  addresses the non-uniqueness and implicitly performs data completion to estimate the missing data $\Ko_{\sph^1\setminus \limset} \signal$. This is equivalent to the selection of the component of the reconstruction in the kernel $ \ker(\Ko_\limset)$. The data completion strongly depends on the chosen regularizer.   The standard Hilbert space norm regularizer $\reg = \norm{\cdot}^2/2$ completes missing data with zero, different regularizers perform non-zero data completion. 

While there are many reasonable choices for the regularizer $\reg$, in this paper we will mainly focus on the $\ell^1$-norm with respect to a suitably chosen frame  and the total variation, each  one coming with its own benefits and shortcomings.  

\subsubsection{Sparse $\ell^1$-regularization}

Let $\synthesis$ denote the synthesis operator of a frame and set $\analysis^\ddagger  \triangleq  (\synthesis\analysis)^{-1} \synthesis$.  In particular, any $\signal \in L^2(\R^2)$  can be written as $\signal = \analysis^\ddagger \analysis \signal $. Synthesis sparsity means that $ \signal = \synthesis \coeff$ where $\coeff$ has only a few non-vanishing entries, whereas  analysis sparsity refers  to  $\analysis u$ having only few  non-vanishing entries. Sparsity can be implemented via regularization using the  $\ell^1$-norm. There are at least two different basic instances of sparse $\ell^1$-regularization namely the synthesis and analysis formulations   
\begin{align} 
   f_{\al, \delta}^{\rm ana} &= \argmin_\signal \frac{1}{2} \norm{\Ko_\limset \signal - \data^\delta}^2 + \alpha \norm{\analysis \signal}_1 \label{eq:ell1analysis}
    \\ 
    f_{\al, \delta}^{\rm syn}  &=   \synthesis  \left( \argmin_\coeff   \frac{1}{2}  \norm{\Ko_\limset (\synthesis\coeff) - \data^\delta}^2  + \alpha \norm{\coeff}_1   \label{eq:ell1synthesis}   \right)\,.
\end{align}
Synthesis and analysis regularization are equivalent in the basis case where they can  be explicitly computed via the diagonal frame decomposition   \cite{goppel2023translation,ebner2023regularization}.
In the  general case synthesis regularization, analysis regularization and  regularization via the diagonal frame decomposition are however fundamentally different \cite{frikel2020sparse}.   

Frame based sparsity constraints have been widely employed for various reconstruction tasks \cite{vandeghinste2013iterative, bubba2018shearlet, candes2000curvelets}.
Note that  theoretical and practical issues  for  general variational regularization can  in particular be applied  to $\ell^1$-regularization. Additionally, $\ell^1$-regularization comes with improved recovery guarantees both in the deterministic and statistical context \cite{candes2002recovering,grasmair2008sparse,lorenz2009convergence}.

\subsubsection{TV regularization}

Total variation regularization is a special case  of variational regularization \cite{scherzer2009variational,acar1994analysis} where the regularizer in \eqref{eq:tik}  is taken as the total variation (TV) 
\begin{equation*} 
	 \abs{u}_{\rm TV}
	 \triangleq
	 \sup\Big\{ \int_{\R^2} u \div v \mid v \in\mathcal{C}_c^1(\R^2,\R^2) \wedge \norm{v}_{2,\infty}\leq 1\Big\}  \,,
\end{equation*}
where  $\norm{v}_{2,\infty}  \triangleq  \sup_x (v_1(x)^2+ v_2(x)^2)^{1/2}$. 
TV regularization  has been proven to well account for missing data in CT image reconstruction \cite{persson2001total,sidky2008image}. 

Using the TV semi-norm as  regularizer tends to smooth out noise while preserving edges within the image. However as for other mono-scale approaches, there is a trade-off between noise reduction and preserving features at specific scales.  Natural images have features across multiple scales which become  either over or under smoothed depending on the particular choice of the regularization parameter \cite{candes2002new,haltmeier2022variational}.  This already has negative impact for fully sampled tomographic systems or simple denoising. To account for the noise a sufficiently large regularization parameter is required that at the same time removes  structures at  small scales.

\subsubsection{Hybrid regularizers}

Hybrid regularizers aim to  combine benefits of the $\ell^1$ regularizer and an additional regularizer such as the TV-seminorm resulting in        
\begin{equation} \label{eq:hybrid}
    \tik_{\alpha, \beta}^{\rm hybrid} (\signal, \data^\delta)
    =
    \frac{1}{2} \norm{\Ko_\limset \signal - \data^\delta}^2 + \alpha \norm{\analysis\signal}_1 + \beta  \reg(\signal) \,.
   \end{equation}
In that context, the sparsity promoting nature of  $\norm{\cdot}_1$ and the data completion property of  $\reg = \abs{\cdot}_{\rm TV}$ are utilized. The $\ell^1$-term  targets a noise-reduced reconstruction and the $\reg$-term targets artifact reduction. 
Various forms of hybrid $\ell^1$-TV regularization techniques have been proposed \cite{vandeghinste2013iterative,kai2020moreau,luo2017image}.  While these methods have been shown to outperform both pure TV and pure $\ell^1$ regularization, they still carry the limitationscof both approaches. 

Minimizing~\eqref{eq:hybrid} has the drawback that the  $\ell^1$-penalty and the TV penalty work against each other in the following sense. The $\ell_1$-norm enforces sparsity of the reconstructed coefficients and for that purpose seeks to recover an image where missing data  completed by values close  to zero. On the other hand, the strength of  TV is to add missing data in a non-vanishing matter. This can be most clearly seen for plain inpainting where forward operator  is given by the restriction $\data_\limset = \signal |_\limset$. If for example $\signal$ is a constant image then filling the missing data  with this constant results in minimal total variation.  This  however works against  the sparsity constraint in a localized frame which aims to fill missing data with small intensity values.    

\section{Complementary  $\ell^1$-TV reconstruction}

We now describe our proposed framework which basically alternates between a reconstruction step and an artifact reduction step inspired by backward backward (BB) splitting. For the following let $\synthesis \colon \Theta \to L^2(\R^2)$ be the synthesis operator  of a frame  (where $\Theta = \ell^2(\Lambda)$) or  a TI-frame (where $\Theta = \ell^2(\Lambda, L^2(\R^2))$).

\subsection{BB splitting algorithm}

Actual implementation of variational regularization\eqref{eq:tik} requires iterative minimization. Splitting methods are very successful in that context. In particular,  BB splitting  applied to the hybrid approach \eqref{eq:hybrid} will be the starting point of our approach. Consider the  splitting     
$\tik_{\alpha, \beta} (\signal, \data^\delta) = \fun_\alpha(\signal, \data^\delta) + \gun_\beta(\signal)$ with  
\begin{align*}
	\fun_\alpha(\signal, \data^\delta) & \triangleq
	\frac{1}{2}  \norm{\Ko_\limset \signal - \data^\delta}^2 
	+ \alpha \norm{\analysis u }_1 \\
	\gun_\beta(\signal) &\triangleq  \beta  \reg(\signal) \,.
\end{align*}
Because $\fun_\alpha(\cdot, \data^\delta)$ and $\gun_\beta$ are both non-smooth, methods that treat both functionals implicitly are an appealing choice. For that purpose one can use the BB splitting algorithm which with coupling constant $\mu>0$ and starting value $\signal^0 \in L^2(\R^2)$ reads     
\begin{align} \label{eq:bb1}
	\aux^{n+1} &\triangleq
	\argmin_h  \fun_\alpha(\aux, \data^\delta) + \frac{\mu}{2}\norm{ \aux  - \signal^n}^2
\\ \label{eq:bb2}
\signal^{n+1} &\triangleq  \argmin_\signal \gun_\beta(\signal)  +  \frac{\mu}{2}\norm{\aux^n - \signal}^2 \,.
\end{align}
The BB splitting algorithm  is  known  to converge to the minimizer of $\fun_\al (\cdot, \data^\delta ) + \beta  \reg_\mu$  where $\reg_\mu(\signal) \triangleq  \inf_w \reg(\signal) + \mu \norm{\signal-w}^2/2$ is the Moreau envelope  of the hybride regularizer \cite{combettes2011proximal}.  

The iterates of the BB splitting algorithm are noise-reduced near solutions of  \eqref{eq:ip} because of  $\fun_\alpha(\aux, \data^\delta)$ in \eqref{eq:bb1}, and regular  because of $ \gun_\beta$  in \eqref{eq:bb2}.  The iterates  $\aux_n, \signal_n$ are coupled via the proximity measure  $\norm{\signal - \aux}^2/2$  resulting in two sequences that are close to each other in the reconstruction domain.

\subsection{Proposed reconstruction framework}

Our algorithm can be  motivated by the BB splitting iteration \eqref{eq:bb1}, \eqref{eq:bb2} utilizing a synthesis  version  for $\ell^1$-minimization  and  TV regularization for the regularizer $\reg$. The main  difference, however,  to the BB iteration is that the proximity term $\norm{\signal-\aux}^2/2$  in the iterative updates are replaced by the data-proximity coupling term $ \norm{\Ko_\limset ( \signal-\aux) }^2/2$.

Our goal is to construct two sequences $(\coeff_n)_{n \in \N}$ and $(\signal_n)_{n \in \N}$ such that $\synthesis \coeff_n$ as well  $\signal_n$ are approximate solutions of $\Ko_\limset \signal = \data^\delta$, however targeting different particular solutions. The reconstruction $\synthesis \coeff_n$ is a noise reduced reconstructions and $\signal_n$ is an  updated  version of  $\synthesis \coeff_n$ targeting reduced limited data artifacts based on $\reg$. To that end define the functionals      
\begin{align*}
	\fun_\alpha(\coeff, \data^\delta) & \triangleq
	\frac{1}{2}  \norm{\Ko_\limset (\synthesis \coeff ) - \data^\delta}^2 
	+ \alpha \norm{\analysis u }_1 \\
	\gun_\beta(\signal) &\triangleq   \signal \mapsto  \beta \abs{\signal}_{\rm TV}  + \mathds{1}_{\geq 0}   \,,
\end{align*}
with $\mathds{1}_{\geq 0}$ being the indicator function of the positive cone given by $\mathds{1}_{\geq 0} (\signal) = 0$ if $\signal \geq 0$ and  $\mathds{1}_{\geq 0} (\signal) = \infty$ otherwise.   

Image reconstruction is done in an iterative fashion similar to \eqref{eq:bb1} however using the data-proximity coupling  $ \norm{\Ko_\limset ( \signal - \aux) }^2/2$. For that purpose we suggest  the iterative procedure      
\begin{align} \label{eq:dd1}
	\coeff^{n+1} &\triangleq
	\argmin_\coeff  \fun_{\al}( \coeff, \data^\delta) + \frac{\mu}{2} \norm{\Ko_\limset( \signal_n - \synthesis \coeff)}^2
\\ \label{eq:dd2}
\signal^{n+1} &\triangleq  \argmin_\signal \gun_{\beta(n)}(\signal)  +   \frac{\mu}{2}\norm{\Ko_\limset( \signal - \synthesis \coeff_n )}^2 \,,
\end{align}
with starting value $\signal^0 \in L^2(\R^2)$.
Here $\norm{\Ko_\limset( \signal - \synthesis \coeff )}^2/2$ is the data-proximity coupling  term and $\mu, \al, \beta(n) > 0$ are parameters. 
The resulting complementary $\ell^1$-TV  reconstruction procedure is summarized in Algorithm~\ref{alg:DC}.

\begin{algorithm}[tbh!]
\caption{Proposed complementary $\ell^1$-TV minimization}
\begin{algorithmic}
\STATE \text{Choose} $\mu, \al, \beta(n) >0$ and $N \in \N$ 
\STATE \text{Initialize}  $f_0  \gets 0$ and $n \gets 0$
\STATE \textbf{repeat}
\STATE \hspace{0.5cm} $\coeff_{n+1}  \gets \argmin_\coeff  \fun_{\al}( \coeff, \data^\delta) + \mu \norm{\Ko_\limset( \signal_n - \synthesis \coeff)}^2/2$
\STATE \hspace{0.5cm} $\signal_{n+1} \gets \argmin_\signal \gun_{\beta(n)}(\signal)  +   \mu \norm{\Ko_\limset( \signal - \synthesis \coeff_n )}^2/2$
\STATE \hspace{0.5cm} $n \gets n + 1$
\STATE \textbf{until} $n \geq N$
\end{algorithmic}
\label{alg:DC}
\end{algorithm}

The proposed steps \eqref{eq:dd1}, \eqref{eq:dd2} in Algorithm~\ref{alg:DC} come with a clear interpretation. The first step \eqref{eq:dd1} is a sparse $\ell^1$-reconstruction  scheme with good noise handling  capabilities.   The second step minimizes the TV norm with the penalty $\norm{\Ko_\limset( \signal - \synthesis \coeff )}^2/2$ and targets artifact reduction. Note that the number $N$ of outer iterations in Algorithm~\ref{alg:DC}  as well as the parameters  $\mu, \al, \beta(n)$ have influence on the final performance. Its  theoretical analysis of the  is interesting and challenging but beyond the scope of this paper.  

\section{Numerical Experiments}
 
In this section we present numerical results using the proposed  Algorithm~\ref{alg:DC} and compare it with standard filtered back projection (FBP),  $\ell_1$-synthesis regularization \eqref{eq:ell1synthesis}, TV regularization and hybrid $\ell^1$-TV regularization \eqref{eq:hybrid}. We consider a limited view as well as a sparse angle scenario and use the NCAT phantom \cite{segars2008realistic}  as image to be recovered (see Figure~\ref{fig:setup}). The NCAT phantom resembles a thorax CT scan, with the spine at the bottom, and ribs on the sides.  The forward and adjoint Radon transforms are  computed using Matlabs standard functions. To mimic real life applications we perturbed the data by Poisson noise with different noise levels corresponding to $10^{a}$ incident photons per pixel bin with $a = 3,4,5$.

\subsection{Implementation details}
    
All minimization problems are  solved with the Chambolle-Pock algorithm \cite{chambolle2011first} using $200$ iterations for $\ell^1$-minimization, and $500$ iterations for TV and hybrid $\ell^1$-TV minimization. This was also the case for the complementary  approach, where for $10^5$ and $10^4$ photon counts we  chose $N=10$ and for $10^3$ photon counts we chose $N=4$ outer iterations. We take the  $n$-th initial value for the $\coeff$ and $\signal$ update as $\coeff_{n-1}$ and  $\signal_{n-1}$, respectively. For  $\analysis$ we use a self-designed TI curvelet transform that in the  case of  limited view data is adapted to the visible wedge; see Appendix~\ref{app:TICT}. Total variation is implemented  as the  $(2,1)$-norm of the discrete gradient computed with finite differences.

The regularization parameters for Algorithm~\ref{alg:DC}  are optimized for $\mu, \alpha, \beta $ with $\beta(n) = 2^n \beta$. Since the described reconstruction techniques rely on good choices for regularization parameters $ \alpha, \beta, \mu$ we perform   systematic parameter sweeps in all cases to obtain optimal reconstructions and a fair comparison. The parameters were optimized in terms of the relative $\ell^2$ reconstruction error $ \norm{\signal_{\rm rec} - \signal}_2 / \norm{\signal}_2$, where  $\signal$ is the true signal and $\signal_{\rm rec}$  the reconstruction. For each parameter and method, we performed a 1D grid search to obtain the lowest $\ell^2$ reconstruction error. In particular, for the proposed complementary $\ell^1$-TV algorithm, we first determine the optimal parameter $\alpha$, and used the optimal choice of the $\coeff$-update as input for the optimization of the parameter $\beta$. All subsequent iterations where then calculated using these parameters.   

For limited view experiments, we chose angular sampling points $\omega(\winkel) = (\cos (\winkel), \sin (\winkel) )$  with  $ \winkel = -\ang{65}, \dots, \ang{64}$ resulting in a total number of 130 directions  covering an angular domain of \ang{130}.  For the sparse view problem we generate Radon data with an angular range of \ang{180}, and a total number of  $50$ angular projections. Photon noise  using  $10^4$ photon counts per bin was added to the data.

\begin{figure}[tbh!] \centering
 \includegraphics[width=0.19\columnwidth,height=0.19\columnwidth]{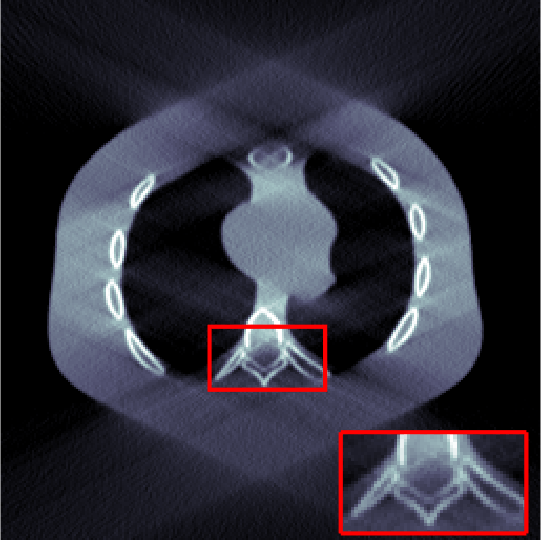}
\includegraphics[width=0.19\columnwidth,height=0.19\columnwidth]{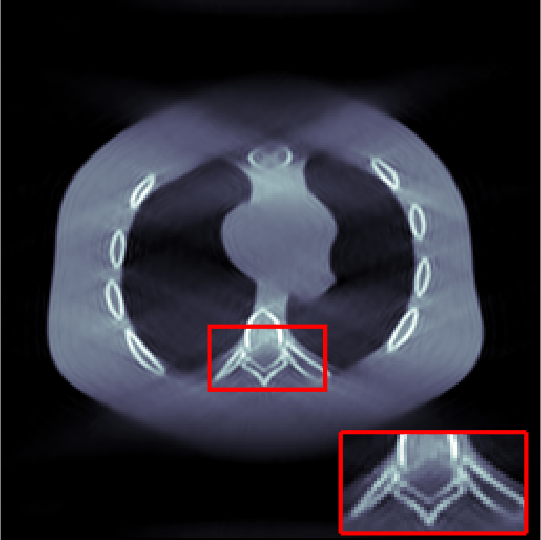}
 \includegraphics[width=0.19\columnwidth,height=0.19\columnwidth]{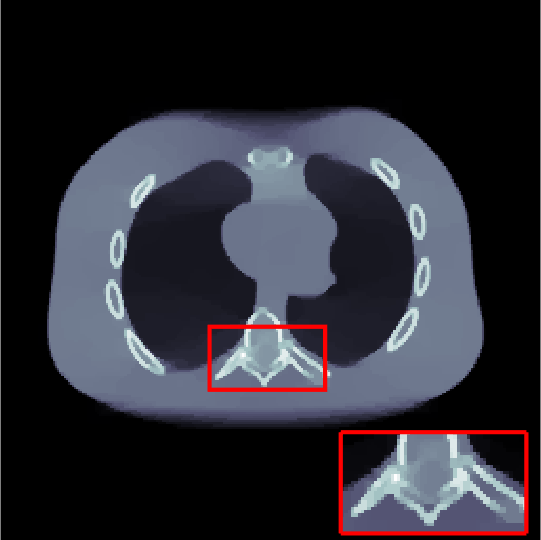}
 \includegraphics[width=0.19\columnwidth,height=0.19\columnwidth]{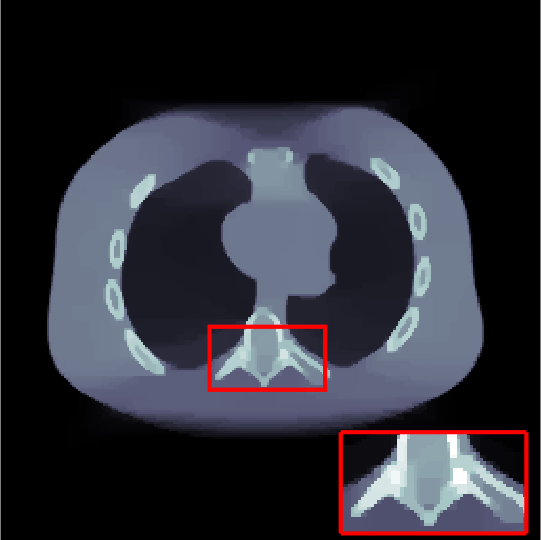}       
\includegraphics[width=0.19\columnwidth,height=0.19\columnwidth]{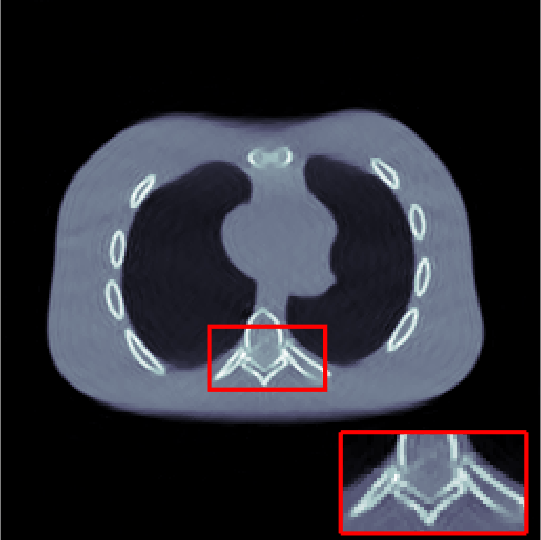}
 \\
\includegraphics[width=0.19\columnwidth,height=0.19\columnwidth]{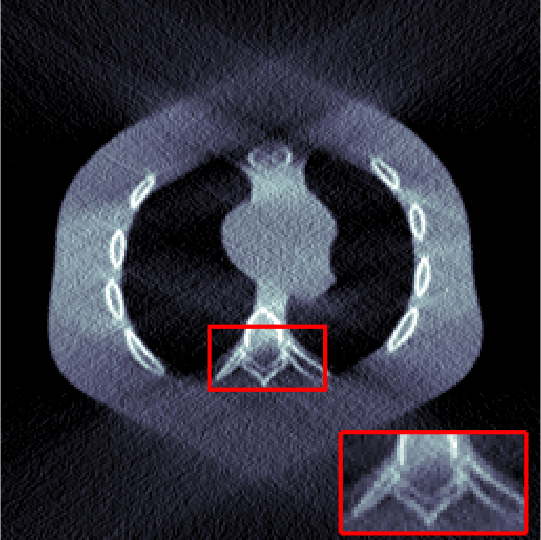}
\includegraphics[width=0.19\columnwidth,height=0.19\columnwidth]{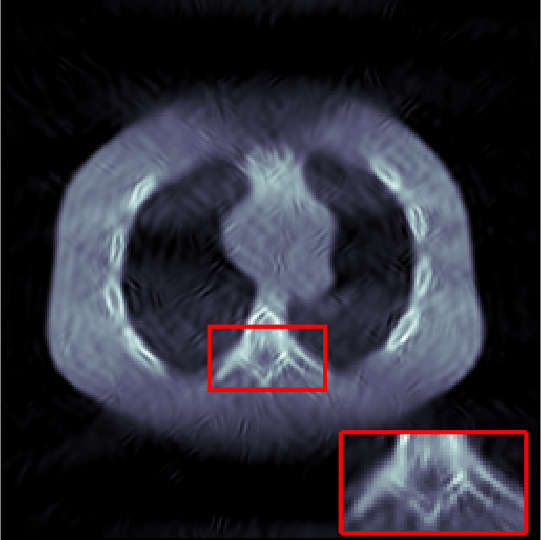}
\includegraphics[width=0.19\columnwidth,height=0.19\columnwidth]{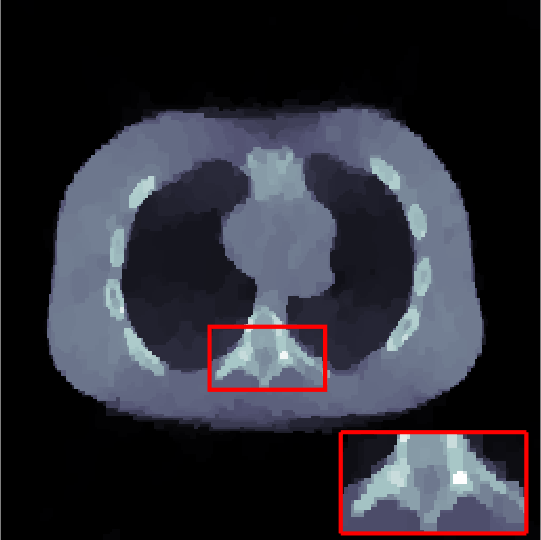}
\includegraphics[width=0.19\columnwidth,height=0.19\columnwidth]{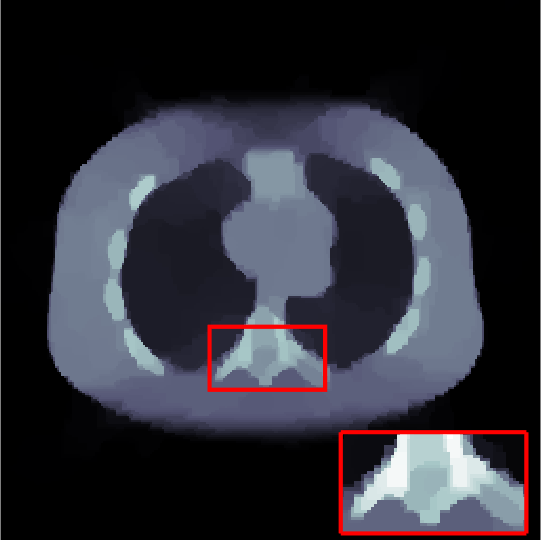}
\includegraphics[width=0.19\columnwidth,height=0.19\columnwidth]{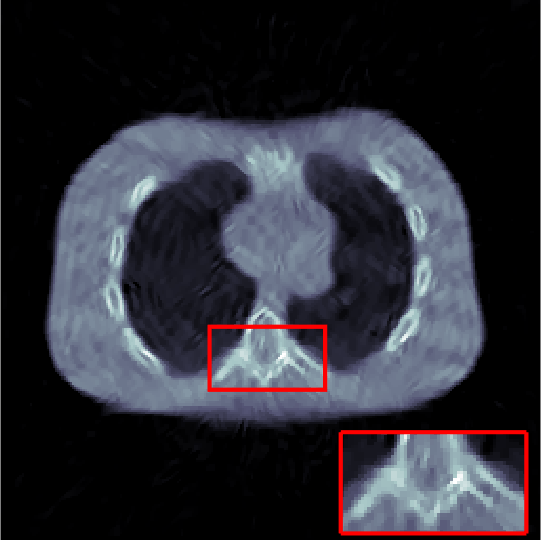} 
\\
\includegraphics[width=0.19\columnwidth,height=0.19\columnwidth]{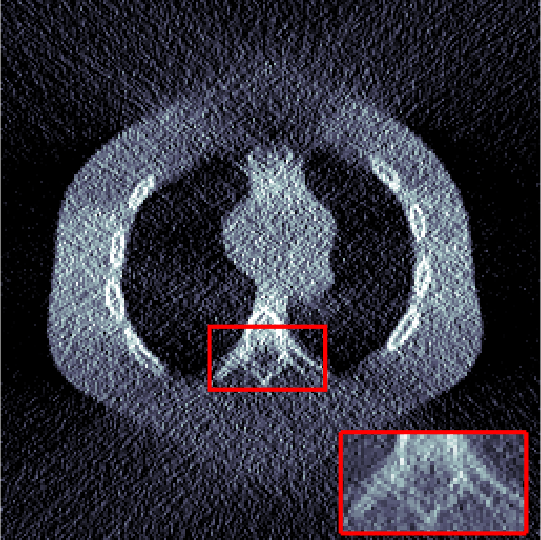}
\includegraphics[width=0.19\columnwidth,height=0.19\columnwidth]{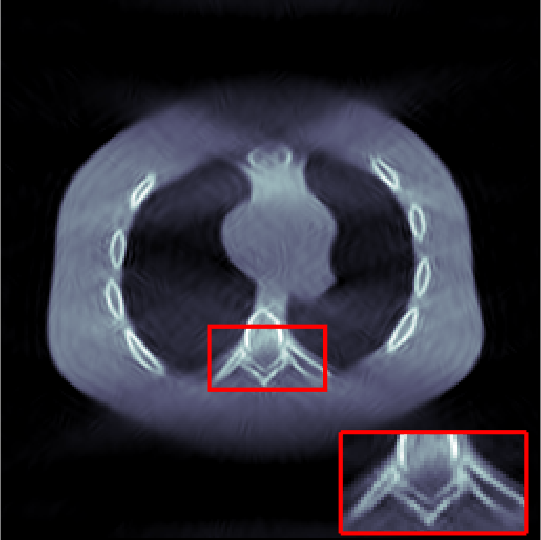}
\includegraphics[width=0.19\columnwidth,height=0.19\columnwidth]{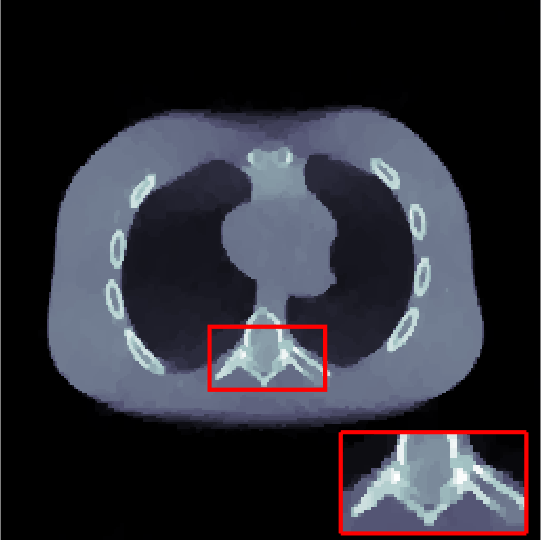}
\includegraphics[width=0.19\columnwidth,height=0.19\columnwidth]{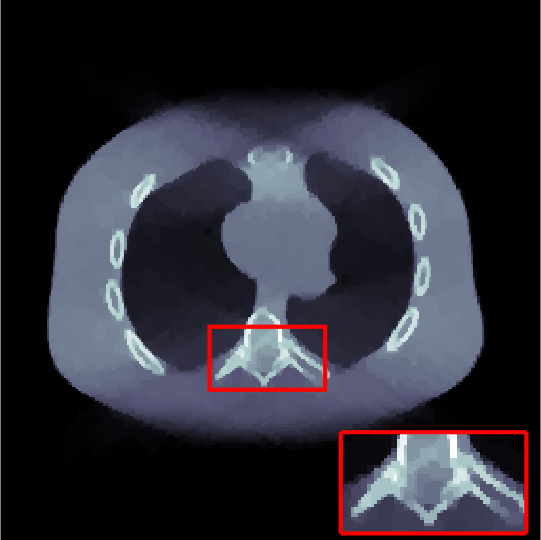}
\includegraphics[width=0.19\columnwidth,height=0.19\columnwidth]{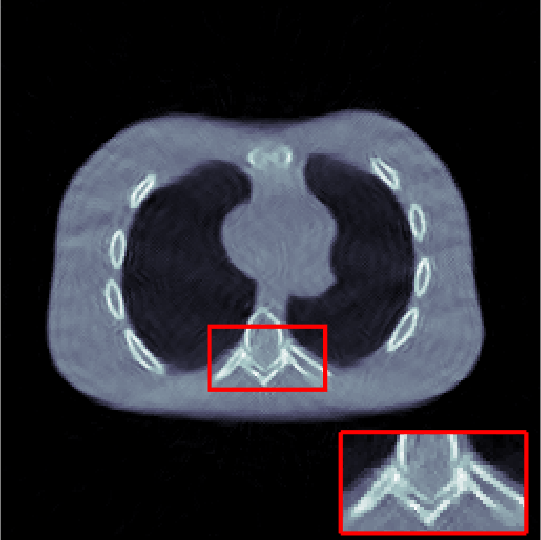}  

\caption{Reconstructions from limited view data.  From left to right, each column shows results using FBP, $\ell^1$-curvelet reconstruction, TV reconstruction, hybrid $\ell^1$-TV and complementary $\ell^1$-TV regularization. Each row corresponds to a different number of $10^5$,  $10^4$ and $10^3$ photon counts. The pixel value range is set to $[0, 1]$ for all images.}
\label{fig:LV}
\end{figure}

\subsection{Results for limited view data}

Figure~\ref{fig:LV} shows reconstruction results for the limited view problem using FBP, $\ell^1$ curvelet reconstruction, TV reconstruction, hybrid $\ell^1$-TV and the proposed complementary $\ell^1$-TV reconstruction.   The results show that the complementary $\ell^1$-TV approach seems to combine the denoising and artifact removing properties of the regularizers in an optimal way.  Taking a closer look at the lowest noise level ($10^5$  photon counts) in the first row, we see that the FBP-reconstruction (left column) and the $\ell^1$-reconstruction (column 2)  suffer from severe  limited view artifacts. While the TV regularized  (column 3) shows less artifacts, we find on the other hand that the fine details of the spine in the magnified part of the image are not reconstructed correctly anymore. This is typical for TV regularization when the regularization parameter  has to be chosen too high in order to address the  noise, resulting in block like artifacts.  A similar observation holds true for the hybrid $\ell^1$-TV reconstruction (column 4). Taking a closer look at the results for the proposed  algorithm (last column) we see that not only are we able to remove the limited view artifacts, but also to recover the fine details accurately. Furthermore, in comparison to the TV reconstruction we  observe that the overall shape of the phantom is better approximated by our approach.

\begin{table}[tbh!]
    \begin{center}
    \begin{tabular}{| l |  l | l | l | l |}
    \toprule 
   \# photons & method &  $\ell^2$-error  & PSNR & SSIM \\
      \midrule
    \multirow{5}{*}{$10^5$}  & FBP & 0.2496 & 17.1021 & 0.2693 \\
    & $\ell^1$ & 0.0756 & 22.590 & 0.559  \\
   & TV & 0.0187 & 29.725 & \textbf{0.953} \\   
   & $\ell^1$-TV & 0.0368 & 25.4124 & 0.8540 \\   
    & proposed & \textbf{0.0103} & \textbf{31.438} & 0.949  \\
    \midrule
   \multirow{5}{*}{$10^4$} & FBP & 0.2719 & 16.7306 & 0.1635 \\
     & $\ell^1$ & 0.0784 & 22.1291 & 0.5430  \\
   & TV & 0.0246 & 27.1590 & \textbf{0.9210} \\   
   & $\ell^1$-TV & 0.0500 & 24.0859 & 0.7633\\ 
    & proposed & \textbf{0.0161} & \textbf{29.0141} & 0.8815 \\
    \midrule
   \multirow{5}{*}{$10^3$} & FBP & 0.4961 & 14.1189 & 0.0696 \\
     & $\ell^1$ & 0.0907 & 21.4974 &  0.4328 \\
   & TV & 0.0411 & 24.9321 & \textbf{0.8621} \\   
   & $\ell^1$-TV & 0.0545 & 23.7100 & 0.7898\\ 
    & proposed & \textbf{0.0311} & \textbf{26.1420} &  0.7906 \\
    \bottomrule
    \end{tabular}
    \end{center}
    \caption{Reconstruction errors for limited view reconstructions.}
    \label{tab:limitedview}
\end{table}

Similar conclusions can be drawn from the  second row of Figure~\ref{fig:LV} showing results for $10^4$ photon counts. Here for the TV and hybrid $\ell^1$-TV regularization even more details are lost. For the other methods, we still have a high level of details visible in the recovered images. However, only for the  proposed method we also obtain an artifact free reconstruction. We attribute the remaining perturbations to the soft-thresholding procedure, that are part of the $\coeff$-update step.  The bottom row in Figure~\ref{fig:LV} shows  the reconstructions using $10^3$  photon counts (the highest noise level  in our experiments). As we see, no method is able to recover the fine structures reliably anymore. However, note that for TV and  hybrid $\ell^1$-TV regularization some of the ribs, which are boundaries of ellipse like structures, now appear to be filled. Simple curvelet-$\ell_1$ regularization and the complementary $\ell^1$-TV approach  still recover the fine holes inside these structures. Again, our method is  capable of removing the limited view artifacts, while also being able to produce a good approximation to the overall shape and details of the phantom.

Summarizing, we can say that our proposed algorithm combines the advantage of both, the denoising capabilities of curvelet-$\ell_1$ regularization, the artifact removal and data recovery properties of the TV regularization approach. A quantitative comparison is given in Table~\ref{tab:limitedview} which compares the reconstructions in terms of the relative $\ell_2$-error, the peak signal-to-noise ratio (PSNR), as well as the structural similarity index measure (SSIM). The best values in each group are highlighted by bolt letters. As we can see, the complementary $\ell^1$-TV approach produces the best reconstructions in terms of the $\ell^2$-error and PSNR, while simple TV regularization is optimal in terms of the SSIM. We find that quantitatively, TV regularization and the complementary $\ell^1$-TV approach are rather similar. However,  qualitatively the advantages of the complementary $\ell^1$-TV method are clearly visible.

\begin{figure}[tbh!]
\centering
\includegraphics[width=0.19\columnwidth,height=0.19\columnwidth]{images/p4_fbp_zoomed.png}
\includegraphics[width=0.19\columnwidth,height=0.19\columnwidth]{images/p4_syn_zoomed.png}  \includegraphics[width=0.19\columnwidth,height=0.19\columnwidth]{images/p4_tv_zoomed.png}
\includegraphics[width=0.19\columnwidth,height=0.19\columnwidth]{images/p4_combined_zoomed.png}
\includegraphics[width=0.19\columnwidth,height=0.19\columnwidth]{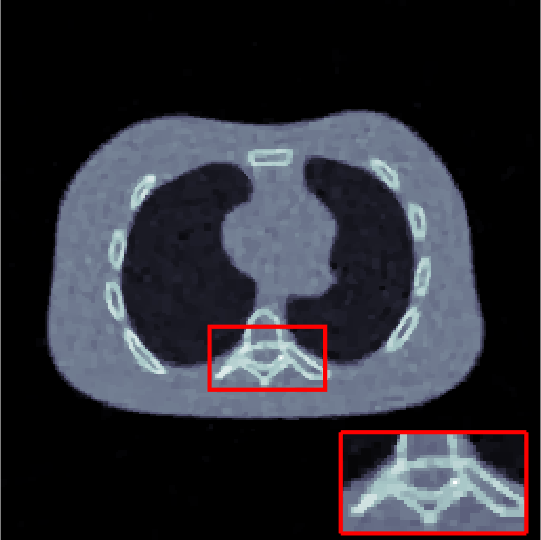}
\caption{Reconstructions from sparse view data using.   From left to right: FBP, $\ell^1$-curvelet,  TV reconstruction, hybrid $\ell^1$-TV, and complementary $\ell^1$-TV. The pixel value range is set to $[0, 1]$ for all images.}
\label{fig:sparse}
\end{figure}

\subsection{Results for sparse view data}

Figure~\ref{fig:sparse} shows reconstruction results for the sparse view problem using FBP, $\ell^1$ curvelet reconstruction, TV reconstruction, hybrid $\ell^1$-TV and the proposed complementary $\ell^1$-TV regularization. We see that all reconstruction methods are able to reproduce the overall phantom rather good. Taking a closer look a the magnified details, we see that the  $\ell^1$-curvelet reconstruction  is able to image the spine rather good. However, we also see that the phantom also suffers from perturbations caused by the soft-thresholding of curvelet coefficients.  The TV regularized reconstruction  one hand does not show severe artifacts, but on the other hand is not able to well recover fine details. Furthermore, some of the inner holes of the ribs start to become filled by the TV regularization, similar to the limited view case. The hybrid $\ell^1$-TV and the proposed complementary $\ell^1$-TV reconstruction on the other hand are  able to incorporate both advantages from curvelet-$\ell_1$ as well as TV regularization. The spine is represented rather well and the phantom does not suffer from curvelet artifacts in both reconstructions.

\begin{table}[tbh!]
    \begin{center}
    \begin{tabular}{| l |  l | l | l | l | }
    \hline 
   \# photons  & method &  $\ell^2$-error  & PSNR & SSIM \\
      \toprule
      \multirow{5}{*}{$10^4$} & FBP & 0.1048 & 20.8702 & 0.1767 \\
     & $\ell^1$ & 0.0136 & 29.7290 & 0.7308 \\
    & TV & 0.0117 & 30.3933 & 0.9294 \\  
    & $\ell^1$-TV & \textbf{0.0080} & \textbf{32.0445} & 0.8884\\
    & proposed & 0.0101 & 31.0289 & \textbf{0.9302} \\
    \bottomrule
    \end{tabular}
    \end{center}
    \caption{Reconstruction errors for sparse view reconstructions.}
    \label{tab:sparse}
\end{table}

A quantitative error assessment is given in Table~\ref{tab:sparse}. Quantitatively, the hybrid $\ell^1$-TV method appears to perform slightly better than the other methods. The visual difference however is quite small and both methods produce equally good reconstructions, where the fine details in the phantom are well represented.

\section{Conclusion}

Similar to many other image reconstruction problems, limited-data CT suffers from instability regarding noise and  non-uniqueness, leading to  artifacts in image reconstruction.  Common regularization approaches use a single regularizer to address both issues, which is accurate for one of the two tasks but not well adapted to the other. To address this issue, in this paper we propose a complementary $\ell^1$-TV algorithm that advantageously combines the denoising properties of $\ell^1$-curvelet regularization and the data completion properties of TV. The main ingredient of our procedure is data-proximity coupling instead of the standard image-space coupling.  

There are many potential future research directions extending  our framework.  We can integrate the data-proximity   coupling into  other splitting type method using proximal terms such as the ADMM algorithm. Further, data-proximity coupling can be combined with preconditioning or other coupling terms. For example, one might replace $\norm{\Ko_\limset( \signal - \synthesis \coeff )}^2$ by $\norm{\mathbf{P}_{\ker (\Ko_\limset)} (\signal - \synthesis \coeff)}$ or may use hard constraints forcing $\Ko_\limset \synthesis \coeff = \Ko_\limset \signal$. Further, one can also  consider general discrepancy functionals $\dist$ in place of the least squares functional  $\norm{\Ko_\limset \signal- \data^\delta}^2/2$.  From the analysis side, studying convergence of iterative procedures as well as  regularization properties is an  important line of future research. Furthermore, a comprehensive investigation of TI-frames for iterative regularization methods would be an interesting research focus. This includes a thorough analysis of theoretical properties along with numerical experiments. In particular, in combination with the limited view CT problem, the study of wedge adapted curvelets, and similar extensions to other limited data problem, could be of high interest.

\section*{Acknowledgments}

The contribution by S.\, G. is part of a project that has received funding from the European Union’s Horizon 2020 research and innovation program under the Marie Sk\l{}odowska-Curie grant agreement No 847476. The views and opinions expressed herein do not necessarily reflect those of the European Commission.

\appendix

\section{Wedge-adapted TI curvelet frames}
\label{app:TICT}
 
Standard curvelets are not well adapted  to limited angle data as  some curvelets elements might may have  small visible components. Our aim  is therefore to construct a curvelet transform that is adapted to the limited view  data $\Ko_\limset$ where  $\limset = \set{(\cos \winkel, \sin \winkel) \mid \winkel \in [-\Phi, \Phi[ }$  for some  $\Phi < \pi/2$. The basic idea is to construct a specific partition of  the frequency plane that  respects the visible wedge $W_\limset = \R \limset$; see left image in Figure~\ref{fig:setup}.    We work with TI variants as the lack of  translation invariance usually results  in visual artifacts \cite{mallat2008wavelet}.  For a recent work on TI-frames in the context of regularization theory see  \cite{goppel2023translation}.

\subsection{Standard TI curvelet frame}

Consider the basic radial and angular Mayer base windows $W \colon [1/2, 2] \to [0,1] $ and $V\colon [-1,1] \to [0,1]$  
\begin{align*}
W (r) &\triangleq 
\begin{cases}
\cos \left( (\pi/2)\nu \left(5-6r\right)\right)  & \text{if } 2/3 \leq r \leq 5/6
\\
1 & \text{if } 5/6 \leq r \leq 4/3
\\
\cos \left((\pi/2) \nu \left(3r -4\right)\right) & \text{if } 4/3 \leq r \leq 5/3
\\
0  & \text{otherwise} \,,
\end{cases}
\\[0.1em]
V(\winkel) &\triangleq  
\begin{cases}
1  & \text{if } \abs{\winkel}\leq 1/3
\\
\cos \left((\pi/2) \nu \left(3 \abs{\winkel} - 1\right)\right), & 1/3 \leq \abs{\winkel} \leq 2/3
\\
0  & \text{otherwise} \,.
\end{cases}
\end{align*}
Here, the auxiliary function $\nu$ is chosen to satisfy $\nu(0) = 0$, $ \nu(1) = 1$  and $\nu(x) + \nu(1-x) = 1$. Possible choices are polynomials, for example $\nu(x) = 3x^2 -2x^3, \nu(x) = 5x^3 - 5x^4 + x^5$ or $\nu(x) = x^4(35-84x+70x^2-20x^3)$. Depending on the choice of $\nu$, the angular windows have smaller or bigger overlap. In this paper we use $ \nu (x) = \chi_{(0,1)} s(x-1)/(s(x-1)+s(x))$ with  $s (x) = \exp (-(1+x)^{-2} -  (1-x)^{-2})$.

The  TI-curvelets are defined  in the frequency space using products  of  rescaled versions of the radial and angular base windows
\begin{equation}\label{eq:TIcurvelets}
    \Fo \boldframe_{j, \ell} (\xi) = 2^{3j/4}  W(2^{-j} r)^2 \cdot  V  \bigl( 2\pi \winkel / N_j  - \ell \bigr)^2 \,,
\end{equation}
where  $\xi =r (\cos \winkel, \sin \winkel)$ and $N_j  \in \N$ and $\La \triangleq \left\{(j,\ell) \mid j \in \N \wedge  \ell \in \{ -N_j/2 , \dots,  N_j/2-1 \} \right\}$.
 At every at scale $j$ the radial window $W(2^{-j} r)$ defines a ring  that is  further partitioned into $N_j$ angular wedges $V  \bigl( 2\pi \winkel / N_j  - \ell \bigr)$.

\begin{theorem} \label{thm:curvelet}
  $(\boldframe_{j, \ell})_{(j, \ell) \in \Lambda}$ is a tight  TI-frame.
\end{theorem}

\begin{proof}
From the definition of the basis windows we have $ \sum_{\ell=-N_j/2}^{N_j/2-1}  V  \bigl( 2\pi \winkel / N_j  - \ell \bigr)^2=1$ and  $\sum_{j\in\Z} \abs{W(2^{-j} r)}^2 = 1$  and therefore  $ \sum_{j,\ell} \abs{\Fo \boldframe_{j,\ell} (\xi)}^2 =1$.
By  the Plancherel identity this is equivalent to the tight frame condition  \eqref{eq:TI-frame} with  $A=B=1$. 
\end{proof}

Curvelet frames are defined by sampling $\boldframe_{j,\ell} \ast \signal$  at points $M_{j,\ell} k$ with a sampling matrix $ M_{j,\ell} \in \R^{2\times 2}$ and sampling index  $k \in \Z^2$. Defining  $\boldframe_{j,\ell,k} \coloneqq \boldframe_{j,\ell} ( x - M_{j,\ell} k)$, this results  in curvelet coefficients $
    \boldframe_{j,\ell} \ast \signal (M_{j,\ell} k)  = \inner{\boldframe_{j,\ell,k}, f} $. The family $  (\boldframe_{j,\ell,k})_{j,\ell,k}$  is a tight frame  which the  associated reproducing formula $\signal = \sum_{j,\ell,k} \inner{\signal, \boldframe_{j,\ell,k}} \overline{\boldframe_{j,\ell,k}}$.
Note that the scale and wedge depending sampling destroys the translation invariance  and the improved denoising property of TI systems \cite{coifman1995translation,goppel2023translation}.

\begin{figure}[tbh!]
    \centering
        \subfloat[]{\includegraphics[width=0.3\columnwidth]{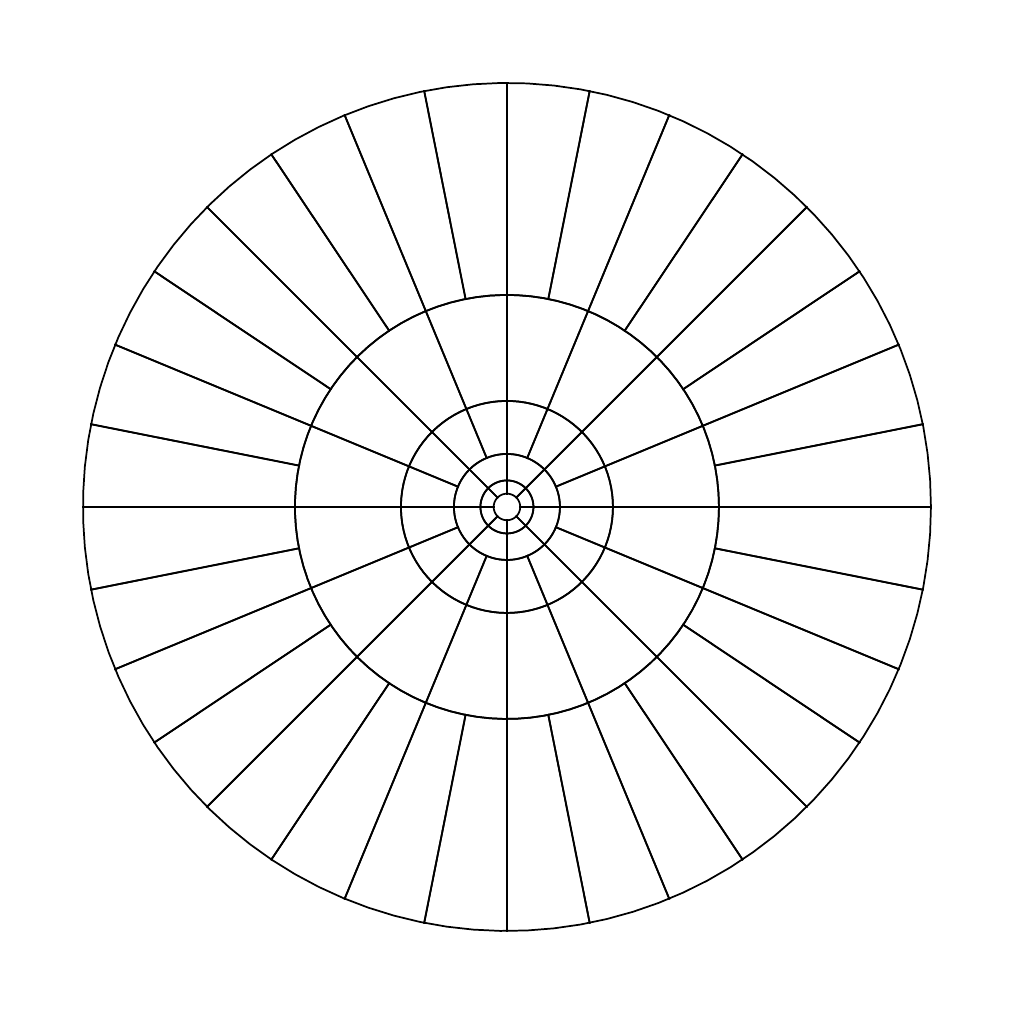}
        \label{fig:tiling1}}
    \hfil
        \subfloat[]{\includegraphics[width=0.3\columnwidth]{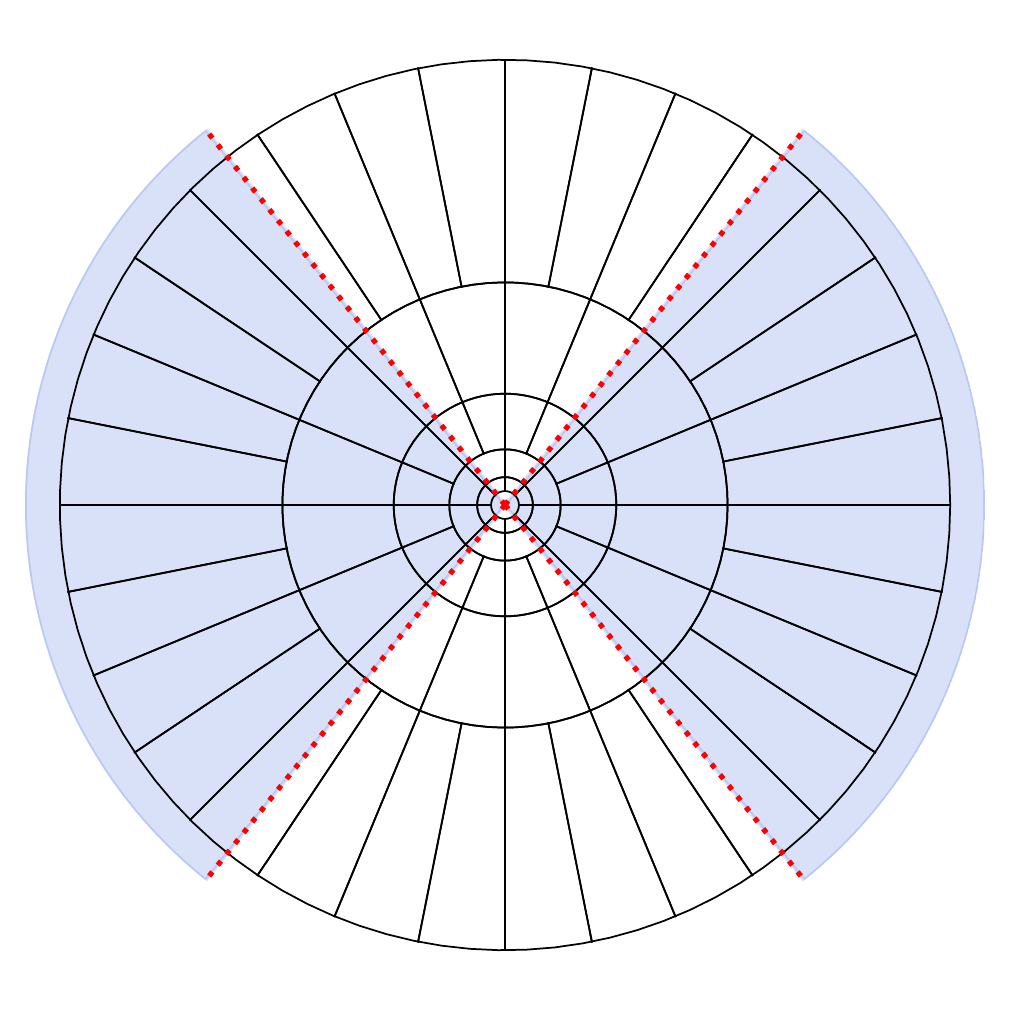}
        \label{fig:tiling-visible}}
    \hfil
        \subfloat[]{\includegraphics[width=0.3\columnwidth]{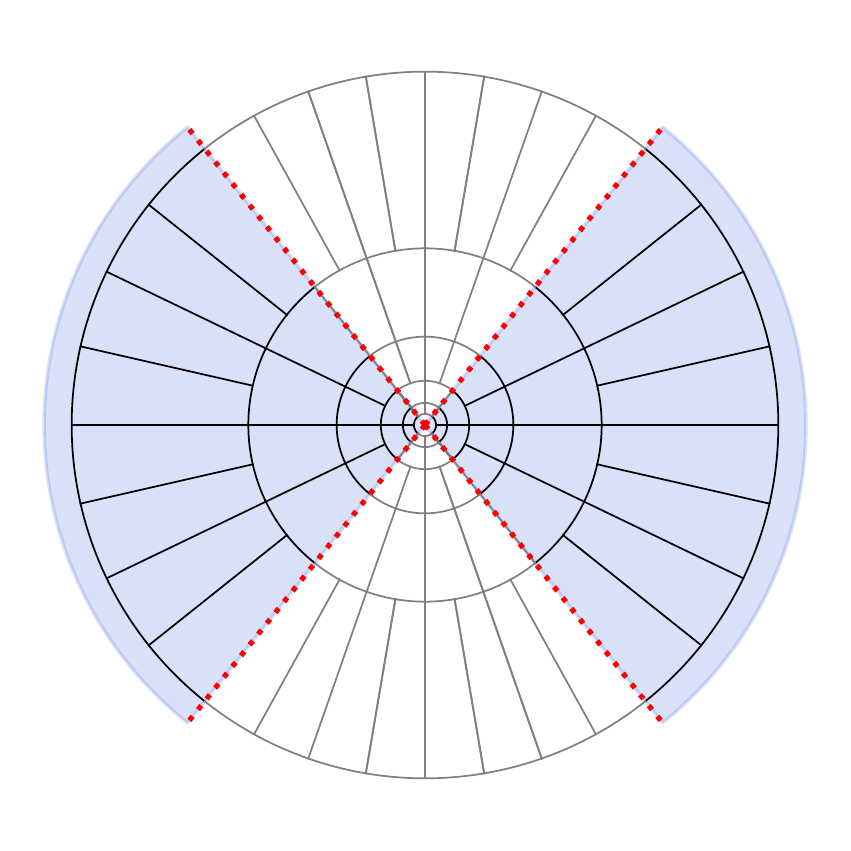}
        \label{fig:tiling-adapted}}
\caption{(a) Standard  curvelet tiling. (b) Visible wedge $W_\limset$ indicated in blue and non-adapted standard curvelet tiling.  (c) Visible wedge $W_\limset$  and wedge adapted tiling.}
\label{fig:tiling}
\end{figure}

\subsection{Wedge adaption}

Due to the limited angular range, the essential support of the Fourier transformed curvelets near the boundary of the visible wedge $W_\limset$ is not fully contained in $W_\limset$; see Figure \ref{fig:tiling-visible}. This results  in an  associated curvelet transform that is not well adapted to the kernel of the limited Radon transform \cite{frikel2013sparse}. In order to adapt  to the visible wedge we modify the  standard  angular  tiling and define two systems $(\boldframe_{j,\ell}^\vis)_{j,\ell}$ and $(\boldframe_{j,\ell}^\inv)_{j,\ell}$  that  we call  the visible and invisible parts of the curvelet  family. For that purpose we define the adjusted angular windows  $V^\vis (\winkel)$ and $V^\inv$ and make sure  that the windows at the boundary sum up  to one.  
Now the wedge-adapted TI curvelets  $\boldframe_{j,\ell}^\vis$, $\boldframe_{j,\ell}^\inv$  are defined as in   \eqref{eq:TIcurvelets} with $V$ replaced by $V^\vis$, $V^\inv$ respectively.   As in  Theorem \ref{thm:curvelet} one shows that the family  $(\boldframe_{j,\ell}^\vis,\boldframe_{j,\ell}^\inv)_{j,\ell} $ forms a TI-frame  of $L^2(\R^2)$. Opposed to the standard TI curvelet frame $(\boldframe_{j,\ell})_{j,\ell} $  it has controlled overlap at the boundary  between visible  and  invisible frequencies.  In a  similar manner we could construct wedge adapted curvelets where we use different numbers $N_j^d$ for each of the four basic wedges. Finally, note that each of  windows has  finite bandwidth. Thus similar to the case of the standard curvelets we  can use Shannon sampling theorem define a wedge adapted curvelet frame by wedge adapted sampling. A detailed mathematical analysis of properties  of its properties is beyond the scope of this paper.

\end{document}